\documentclass[12pt,oneside,reqno]{amsart}
\usepackage[margin=1in, marginparwidth=0.8in, marginparsep=0.1in]{geometry}
\usepackage{enumitem, verbatim}
\usepackage{microtype}
\usepackage{physics}
\usepackage{xcolor}
\usepackage{bbold}

\allowdisplaybreaks

\usepackage{graphicx}
\usepackage{hyperref}
\usepackage{color}
\hypersetup{colorlinks, linkcolor=red, citecolor=cyan}
\usepackage{amsthm}
\usepackage[nameinlink]{cleveref}
\usepackage{amsmath,amssymb}
\usepackage{float}
\usepackage{bbm}
\usepackage{standalone}
\usepackage{relsize}
\usepackage{mathdots}
\usepackage{empheq}
\usepackage{stmaryrd}
\usepackage{fontawesome5}

\newcommand{\mapleleaf}{\faCanadianMapleLeaf}

\usepackage[most]{tcolorbox}
\usepackage{xcolor}

\newtcolorbox{colourful}{
  colback=red!5,
  colframe=red!45!black,
  boxrule=0.5pt,
  sharp corners,
  left=4pt,
  right=4pt,
  top=4pt,
  bottom=4pt,
  before skip=1em,
  after skip=1em
}

\usepackage{enumitem}

\definecolor{ProofRed}{RGB}{145,30,35}
\definecolor{ProofLightRed}{RGB}{255,244,244}

\newtcolorbox{claimbox}{
    enhanced,
    breakable,
    colback=ProofLightRed,
    colframe=ProofRed,
    boxrule=0.6pt,
    arc=0pt,
    left=8pt,
    right=8pt,
    top=6pt,
    bottom=6pt,
    before skip=8pt,
    after skip=8pt
}

\usepackage{tikz, tikz-cd, tikz-3dplot}
\usetikzlibrary{calc}

\makeatletter
\newcommand{\namedlabel}[2]{%
  \phantomsection
  \def\@currentlabel{#1}%
  \label{#2}%
}
\makeatother

\newtheorem{thm}{Theorem}
\newtheorem{lem}[thm]{Lemma}
\newtheorem{prop}[thm]{Proposition}
\newtheorem{claim}{Claim}
\newtheorem{cor}[thm]{Corollary}

\theoremstyle{definition}

\newtheorem{ntn}[thm]{Notation}
\newtheorem{obs}[thm]{Observation}

\newtheorem{rmk}[thm]{Remark}

\newcommand{\N}{\mathbb{N}}

\newcommand{\A}{\mathcal{A}}
\newcommand{\B}{\mathcal{B}}

\title{Uniform Factor-Balancedness in Arnoux--Rauzy Words}

\author{Halyna Bowley}

\date{September 2026}

\begin{document}

\begin{abstract}
We characterize uniform factor-balancedness in Arnoux--Rauzy words, generalizing the work of Espinoza, Popoli, and Stipulanti on ternary alphabets to finite alphabets of any size. In particular, we prove that an Arnoux--Rauzy word $\mathbf{x}$ is uniformly factor-balanced if and only if $\mathbf{x}$ has bounded weak partial quotients.
\end{abstract}

\maketitle

\section{Introduction}
\textit{Arnoux--Rauzy} words were defined in 1991 as a generalization of binary Sturmian words to larger alphabets \cite{Arnoux1991}. They can be described by considering the limits of sequences of the following substitutive maps. For a fixed letter $a$ in a finite alphabet $\A,$ we define the \textit{Arnoux--Rauzy substitution} $\sigma_a$ by 
$$\sigma_a(b) := \begin{cases} ba ~~&\text{if}~ b \in \A \setminus \{a\},\\
a~~~&\text{otherwise.}
\end{cases}$$We extend to a map on words by applying $\sigma_a$ letter by letter. 

Let $(a_n)_{n \geq 0}$ be a sequence of letters in $\A$ such that each letter in $\A$ occurs infinitely often in the sequence. Arnoux--Rauzy words are exactly those that have the same set of factors (finite consecutive subwords) as an infinite word $\mathbf{y}$ generated by the limit $$\mathbf{y}:= \lim_{n \to \infty} \sigma_{a_0} \circ \sigma_{a_1} \circ \cdots \circ \sigma_{a_{n}}(a).$$ This factor-equivalent characterization lends itself well to the study of various balancedness properties due to the invariance of balancedness among words with identical factor sets. In this paper, we are specifically interested in studying uniform factor-balancedness. We say that an infinite word $\mathbf{x}$ is \textit{uniformly factor-balanced} if there exists an integer $C > 0$ such that for all finite words $w,$ and all factors $u,v$ of $\mathbf{x}$ of the same length, the difference between the number of occurrences of $w$ in $u,v$ is bounded by $C.$

It turns out that an Arnoux--Rauzy word's generating sequence often controls the frequencies at which factors can occur. For a sequence $(a_n)_{n\geq 0},$ write $a_0a_1a_2 \dots = b_0^{k_0}b_1^{k_1}b_2^{k_2} \dots$ by grouping together maximal runs of equal consecutive letters. The derived sequence $(k_m)_{m \geq0}$ of positive integers is called the sequence of \textit{weak partial quotients}. Restricting to alphabets of size at most three, Espinoza, Stipulanti, and Popoli showed in 2026 that the properties of uniform factor-balancedness and having bounded weak partial quotients are equivalent \cite{espinoza2026factorbalancednesslinearrecurrencefactor}. It is natural to conjecture that this result extends to words over any finite alphabet. The main undertaking of this paper is showing that this conjecture is the truth.

\begin{thm}\label{thm:unif_fb}
    Let $\A$ be any finite alphabet. An Arnoux--Rauzy word over $\A$ is uniformly factor-balanced if and only if it has bounded weak partial quotients.
\end{thm}

\subsection{A few words on words.} An \textit{alphabet} is a finite set $\A$ whose elements we call \textit{letters}. A \textit{word} over $\A$ is a (potentially infinite) concatenation of letters $w = w_1w_2\cdots$, where each $w_i$ is in $\A.$ We say $w$ has length $|w| = n$ if it consists of $n$-many letters. A finite word $w = w_1 \cdots w_n$ is a \textit{factor} of another (potentially infinite) word $x = x_1x_2 \cdots$ if there exists $i \geq 1$ such that $w_1 \cdots w_n = x_i \cdots x_{i+n-1}.$ In this case, we say that the index $i$ is an \textit{occurrence} of $w$ in $x$, and denote by $|x|_w$ the number of occurrences of $w$ in $x.$ 

Let $\mathbf{x}$ be an infinite word and $w$ a finite word. Fix a constant $C >0.$ We say that $w$ is \textit{C-balanced} in $\mathbf{x}$ if, for all equal-length factors $u,v$ of $\mathbf{x},$ we have that $||u|_w - |v|_w| \leq C.$ We say that $\mathbf{x}$ is \textit{C-letter-balanced} if every letter $a$ is $C$-balanced in $\mathbf{x}.$ More broadly, we say $\mathbf{x}$ is \textit{factor-balanced} if, for any finite word $w,$ there exists a constant $C_w$ such that $w$ is $C_w$-balanced in $\mathbf{x}.$ If the constant does not depend on $w,$ then we get the definition of uniform factor-balancedness given on the previous page. In the case of Arnoux--Rauzy words, factor-balancedness is equivalent to letter-balancedness (but importantly, not to uniform factor-balancedness) \cite{BERTHE201993}. 

Finally, we note that the substitutive characterization of Arnoux--Rauzy words given in the introduction is not their original definition. An infinite word $\mathbf{x}$ over an alphabet of size $d$ is defined to be Arnoux--Rauzy if all of its factors occur infinitely often, it has exactly $(d-1)n + 1$ factors of length $n,$ and has exactly one left special and one right special factor of length $n$. This definition is equivalent to the previously stated substitutive characterization \cite{Arnoux1991}; in this paper we exclusively use the latter. 

\subsection{History and motivation}
The study of Arnoux--Rauzy words was motivated by that of Sturmian words, the class of binary aperiodic infinite words of lowest possible factor complexity (in particular, those words that have exactly $n+1$ factors of length $n$). The property of being Sturmian is closely related to balancedness: one can prove that a binary word is Sturmian if and only if it is aperiodic and $1$-balanced \cite{DBLP:journals/mst/CovenH73}. In 1940, Morse and Hedlund proved that Sturmian sequences are natural codings of irrational rotations of the circle \cite{Morse1940SymbolicDI}. Later, Arnoux and Rauzy proved that Morse and Hedlund's work connects to an interpretation of Sturmian sequences through continued fraction algorithms \cite{Arnoux1991}.

Arnoux--Rauzy words were defined in 1991 as larger-alphabet analogues of Sturmian words in order to extend these correspondences beyond the binary setting \cite{thuswaldner2020sadics}. For example, it was conjectured that every ternary Arnoux--Rauzy word was a natural coding of a translation of the two-dimensional torus \cite[Introduction]{Cassaigne2000}. However, these efforts were stymied by the fact that certain balancedness properties fail to hold over larger alphabets. In contrast to the strict $1$-balancedness of Sturmian words, Cassaigne, Ferenczi, and Zamboni proved the existence of ternary Arnoux--Rauzy words that are imbalanced for all finite $C$ \cite{Cassaigne2000}. Moreover, they used this imbalancedness result to disprove the conjectured correspondence with toral rotations. 

Subsequently, balancedness drew interest as a metric of how far a given Arnoux--Rauzy word strays from its nicely-behaved Sturmian archetype (see \cite{thuswaldner2020sadics} for a general overview and \cite{BERTHE201993} and \cite{berthe2021multidimensionalcontinuedfractionssymbolic} for more specific results regarding the relation between balancedness and dynamics). In the aforementioned paper demonstrating imbalancedness, Cassaigne, Ferenczi, and Zamboni note that every linearly recurrent Arnoux--Rauzy word is letter-balanced (and therefore, as mentioned earlier, factor-balanced). 

Obtaining a more precise characterization of balancedness, however, has proven somewhat tricky. In 2013, Berth\`e, Cassaigne, and Steiner were able to show that, over a ternary alphabet, an Arnoux--Rauzy word is letter-balanced if its weak partial quotients are bounded \cite{BCS}. Shortly after, Delecroix, Hejda, and Steiner gave a sufficient criterion for letter-balancedness over arbitrary alphabets: namely, they proved that if there exists a constant $h$ such that $\{a_{n}, \dots, a_{n+h}\} = \A$ for all $n$, where $(a_n)_{n \geq0}$ generates the Arnoux--Rauzy word $\mathbf{x}$, then $\mathbf{x}$ is letter-balanced \cite{Delecroix_2013}. However, requiring the existence of such an $h$ is much stricter than bounding weak partial quotients. They show in the same paper that almost every Arnoux--Rauzy word is letter-balanced. In a 2026 preprint, Espinoza, Popoli, and Stipulanti were able to employ the aforementioned result from \cite{BCS} to show that an Arnoux--Rauzy word over a ternary alphabet has bounded weak partial quotients if and only if it is uniformly factor-balanced \cite{espinoza2026factorbalancednesslinearrecurrencefactor}. This result in \cite{espinoza2026factorbalancednesslinearrecurrencefactor} is a special case of Theorem~\ref{thm:unif_fb}.

\subsection{Our methods.}\label{ss:our_method} A central ingredient in the proof of Espinoza, Popoli, and Stipulanti's result in \cite{espinoza2026factorbalancednesslinearrecurrencefactor} is the fact that, over ternary alphabets, an Arnoux--Rauzy word having bounded weak partial quotients implies that it is letter-balanced \cite{BCS}. In particular, as we explain in \S\ref{sec:main_thm}, this is the only step in their proof that requires a restriction on the alphabet's size. The key to proving Theorem~\ref{thm:unif_fb} is, consequently, a generalization of this letter-balance statement to alphabets of any size.

\begin{thm}\label{thm:lb}
     Let $\mathbf{x}$ be an Arnoux--Rauzy word over an alphabet of size $d$ whose weak partial quotients are bounded by $h.$ Recursively, define $$R_2:= 0, \quad R_{N+1}:= (h+3)R_N + 2h + 13.$$ Set
     
     $$C_{d,h} := \left\lceil\frac{R_d +2}{d-1}\right\rceil.$$ Then $\mathbf{x}$ is $(C_{d,h})$-letter-balanced.
\end{thm}

\begin{rmk}
    For intuition on the definition of the sequence $(R_N)_{N \geq 2},$ see Subsection \ref{sec:outline_big} (the outline of the proof of the main lemma). The recursive definition is most natural for the proof of Lemma~\ref{lem:big}, but we also have two closed-form expressions 
    $$R_N =\frac{2h+13}{h+2}((h+3)^{N-2} -1);$$$$C_{d,h} =\left\lceil \left(\frac{(2h +13)((h+3)^{d-2}-1)}{h+2} + 2\right) \frac{1}{d-1}\right\rceil = O\left(\frac{(h+3)^{d-2}}{d}\right).$$
\end{rmk}
\bigskip

The methods used in \cite{BCS} to show letter-balancedness in the ternary setting do not readily generalize: the proof is structured around casework that quickly becomes unmanageable over alphabets with as few as four or five letters. We therefore develop a more robust mechanism for controlling letter discrepanies that does not depend on the size of the alphabet. This is Lemma~\ref{lem:big}, which we state in \S\ref{sec:main_thm} and prove in \S\ref{sec:big}. The governing heuristic of this lemma is that an excessive letter discrepancy concentrated in one letter propagates to a collective discrepancy over a growing set of letters as we ``unpeel" successive layers of $\sigma_a$'s. This action of ``unpeeling" is the desubstitution described in the following section.

\section{Desubstitution and basic identities}\label{sec:identities}
We start by introducing the following simple desubstitution property of the $\sigma_a$'s, along with useful related notation.

\begin{ntn}\label{notation:1}
    As noted in \cite{BCS}, for any factor $u$ of $\sigma_a(\mathbf{x})$ for some word $\mathbf{x}$ (either finite or infinite), there exists a unique factor of $\mathbf{x}$, which we denote by $u^{(1)}$, that satisfies

\begin{equation*}
a^{\delta}\,\sigma_{a}(u^{(1)}) = u\,a^{\varepsilon},
\end{equation*} where

\[
\delta =
\begin{cases}
1 & \text{if } u \text{ starts with } a,\\
0 & \text{otherwise};
\end{cases}
\qquad
\varepsilon =
\begin{cases}
0 & \text{if } u \text{ ends with } a,\\
1 & \text{otherwise}.
\end{cases}
\]
    
    For a factor $u$ of a word $\sigma_{a_1} \circ \cdots \circ \sigma_{a_n}(\mathbf{x}),$ denote by $u^{(n)}$ the factor of $\mathbf{x}$ obtained after applying this desubstitution $n$ times. Analogously, given an Arnoux--Rauzy word $\mathbf{x}$ generated by the sequence $(a_n)_{n \geq 0},$ we denote by $\mathbf{x}^{(m)}$ the Arnoux--Rauzy word generated by the truncated sequence $(a_n)_{n \geq m}.$ We also introduce several abbreviations for convenience's sake. Let
    \begin{align*}
        L_{u,v}^{(n)} &:= |u^{(n)}| - |v^{(n)}|;\\
        d_{b,u ,v}^{(n)} &:= |u^{(n)}|_b - |v^{(n)}|_b;
    \end{align*} denote the length and letter discrepancy, respectively, of the $n$th preimages of a fixed pair of factors. In the same vein, for a fixed subset $\B \subset \A,$ we denote by

    \begin{align*}
        d_{\B,u,v}^{(n)} &:= \sum_{b \in \B} d_{b, u, v}^{(n)};\\
        \rho^{(n)}(\B)_{u,v} &:= d_{\B,u,v}^{(n)} - L_{u,v}^{(n)} = -d_{\A
        \setminus \B, u, v}^{(n)},
    \end{align*} the collective letter discrepancy of letters in $\B$ and the complement of $\B,$ respectively, of the $n$th preimages. We drop the subscript $u,v$ when it is clear from context.
\end{ntn}

The following lemma describes how desubstitution affects factor length and letter occurrence. 

\begin{lem}[Lem 17 in \cite{BCS}]\label{lem:delta} Let $u$ be a factor of $\sigma_{i}(\mathbf{x})$ for some word $\mathbf{x}.$ Then there exists $\delta_u \in \{-1, 0, 1\}$ such that 

\[
|u^{(1)}|
= |u|_{i} + \delta_u;
\]

\[
|u^{(1)}|_j =
\begin{cases}
|u|_j & \text{if } j \in \mathcal{A} \setminus \{i\}, \\[1ex]
2|u|_{i} - |u| + \delta_u & \text{if } j = i.
\end{cases}
\] Furthermore, $|u^{(1)}| < |u|$ if $|u| \ge 2$, and
$|u^{(1)}| \in \{0,1\}$ if $|u| \in \{0,1\}$.   
\end{lem}

Finally, we record several identities for ease of reference later on.

\begin{obs} Fix an Arnoux--Rauzy word with generating sequence $(a_n)_{n \geq 0}$ and factors $u,v$. As above, we omit the subscript when it is clear from context. Then we obtain the following basic equalities:

\begin{align}
    L^{(\ell)} &= \sum_{a \in \mathcal{A}} d_a^{(\ell)}\label{eq:sum};\\
    d_{a_\ell}^{(\ell)} - L^{(\ell)} &= d_{a_\ell}^{(\ell + 1)} - L^{(\ell + 1)}\label{eq:stable:diff};\\
    L^{(\ell+1)} &= d_{a_\ell}^{(\ell)} + \delta_u - \delta_v \label{eq:L};\\
    d_{a_\ell}^{(\ell+1)} &= 2d_{a_\ell}^{(\ell)} - L^{(\ell)} + \delta_u - \delta_v \label{eq:d} ;\\
    d_{a_\ell}^{(\ell + 1)} - d_{a_\ell}^{(\ell)} &= d_{a_\ell}^{(\ell)} - L^{(\ell)} + \delta_u - \delta_v \label{eq:diff}.
\end{align}
Identities (\ref{eq:sum}) and (\ref{eq:stable:diff}) follow from the definitions, and (\ref{eq:L})-(\ref{eq:diff}) are direct applications of Lemma~\ref{lem:delta}.
\end{obs}

\section{Proofs of the main theorems}\label{sec:main_thm}

In the following, we assume $|\A| = d \geq 3.$ When $|\A| = 2$, the families of Arnoux--Rauzy and Sturmian words coincide. Words in the latter family are known to be uniformly factor-balanced \cite{espinoza2026factorbalancednesslinearrecurrencefactor}.

The proof of Theorem~\ref{thm:lb} proceeds by contradiction. We assume there exist a pair of equal-length factors $u,v$ and a letter $z$ such that $z$ has discrepancy $C > C_d$. Ultimately, we show that this implies that for all indices $\ell >0$, there exists a set of letters with collective discrepancy between $u^{(\ell)}, v^{(\ell)}$ greater than one. 

Our first observation is that, for any letter $a_i \in \A,$ the discrepancy $d_{a_i}^{(i)}$ can only go down after desubstituting by $\sigma_{a_i}$ if the length discrepancy $L^{(i)}$ is strictly greater than $d_{a_i}^{(i)} - 2$ (see (\ref{eq:d})). We consider what happens at the greatest index $i$ such that $d_{z}^{(i)}$ is greater than the threshold $C$ and $L^{(i)}$ is less than $C$. Relabeling $u^{(i)},v^{(i)}$ to $u,v$, this maximal setting is formalized in the following condition. 

\begin{colourful}\namedlabel{\protect\mapleleaf}{cond:maple}
    \textbf{Condition} \textcolor{red}{\faCanadianMapleLeaf}.
    For every $\ell > 0,$  there do not exist a letter $z \in \A$ and factors $u',v'$ of $u^{(\ell)}, v^{(\ell)},$ respectively, such that $$|u'| - |v'| < C <|u'|_z-|v'|_z.$$
\end{colourful}

\begin{rmk}\label{rmk:maple}
    If \ref{cond:maple}  ~holds, then we cannot have both 

    \begin{equation}
    d_z^{(\ell)} - L^{(\ell)} \geq 2\quad \text{and}\quad d_z^{(\ell)} \geq C+1 \label{eq:maple_alt}. 
    \end{equation} Otherwise, we could truncate $u^{(\ell)}$ to obtain a pair of factors that contradicts \ref{cond:maple}. Often, we will obtain a contradiction by showing that (\ref{eq:maple_alt}) holds.
\end{rmk}

From (\ref{eq:L}), the length discrepancy $L^{(i)}$ is determined by the letter discrepancy of $a_{i-1}$ at stage $i-1.$ In particular, this means that a large length discrepancy can only come about due to a large letter discrepancy. This observation gives us the heuristic that the fall of one letter's discrepancy implies the existence of another letter with a large discrepancy. A more rigorous version of this statement is the content of the following lemma.

\begin{lem}\label{lem:break}
    Fix an integer $C >0.$ Suppose $u,v$ are factors of an Arnoux--Rauzy word such that $$|u| - |v| \leq C-1 \quad \text{and}\quad  |u|_z - |v|_z \geq C+1$$ for some $z \in \A$. Then there exists an index $\ell$ and a letter $b\in \A \setminus\{z\}$ such that
    \begin{equation}\label{eq:break}
        d_z^{(\ell)} \geq C+1;\quad d_b^{(\ell)} \geq C-2;\quad \rho^{(\ell)}(\{z,b\}) \geq C.
    \end{equation}
\end{lem}

\begin{proof}
     We start by showing that there must exist an index $\ell$ such that 
     \begin{equation}\label{eq:first_sentence}
          L^{(\ell)} \geq C; \quad d_z^{(\ell)} \geq C+1; \quad a_{\ell - 1} \ne z,
     \end{equation} where $a_{\ell-1}$ is the $(\ell-1)$th letter in the generating sequence.

     %pre-claude version: Then suppose that we have $a_{\ell-1} \ne a_\ell = \cdots = a_{\ell+n-1} \ne a_{\ell+n}$ with $a_\ell = j$ and that $d_j^{(\ell)} \geq C+1.$ By (\ref{eq:diff}), at each intermediary index $\ell \leq i \leq \ell + n-1,$ we have 
    
    %$$d_j^{(i+1)} - d_j^{(i)} \geq d_j^{(\ell)} - L^{(\ell)} -2 \geq 0.$$

    %In particular, $d_j^{(\ell+n-1)} \geq C+1.$ Inductively, this implies that $d_j^{(p)} \geq C+1$ for all indices $p.$
    
      We assume otherwise, and argue inductively that this implies that $d_z^{(p)} \geq C+1$ for all indices $p.$ By hypothesis, $d_z^{(0)} \geq C+1.$ Moreover, if $d_z^{(p)} \geq C+1$ at $p$ such that $a_{p} \ne z,$ then, since $d_z^{(p)} = d_z^{(p+1)},$ we also have $d_z^{(p+1)} \geq C+1$. Suppose now that $a_p = z,$ and this occurs during a run of $z$'s $a_{j-1} \ne a_{j} = \cdots =a_p = \cdots = a_{j+n-1} \ne a_{j+n}$. By our inductive hypothesis, $d_z^{(j)} \geq C+1.$ Therefore $L^{(j)} \leq C-1,$ because otherwise the index $j$ would satisfy (\ref{eq:first_sentence}). By (\ref{eq:stable:diff}) and (\ref{eq:diff}), at each intermediary index $j \leq i \leq j + n-1,$ we have 
    
    $$d_z^{(i+1)} - d_z^{(i)} \geq d_z^{(j)} - L^{(j)} -2 \geq 0.$$ In particular, $d_z^{(p+n)} \geq C+1.$ Therefore, we can conclude that $d_z^{(p)} \geq C+1$ for all indices $p.$ This contradicts that there exists $p$ such that $u^{(p)}$ has length at most $1.$

    Suppose now that $\ell$ is the smallest index satisfying (\ref{eq:first_sentence}), and let $b$ be the letter occurring at index $\ell-1.$ We first note that (\ref{eq:L}) implies $d_b^{(\ell-1)} \geq L^{(\ell)} - 2 \geq C-2.$ We then argue that $d_{\A \setminus\{b,z\}}^{(\ell-1)} \leq -C$ by considering two cases: either $b$ is preceded by a run of $z$'s or a letter $c \in \mathcal{A} \setminus \{z\}.$ In the first case, note that $d_z^{(\ell -n-1)}  \geq C +1,$ where $n$ is the length of the run of $z$'s. Otherwise, we can choose the smallest index $i$ such that $d_z^{(i+1)} < C+1$, and let $m$ be the length of the (possibly empty) run of $z$'s preceding $i$. By the same reasoning as above, we must have $ 1 \geq d_z^{(i)} - L^{(i)} = d_z^{(i-m)} - L^{(i-m)},$ which implies $L^{(i-m)} \geq C$ by minimality of $i.$ The index $i-m$ then satisfies (\ref{eq:first_sentence}), contradicting the minimality of $\ell.$ The minimality of $\ell$ then implies that $L^{(\ell - n -1)} \leq C-1.$ Therefore $$d_z^{(\ell-1)} - L^{(\ell-1)} = d_z^{(\ell-n-1)} - L^{(\ell-n-1)} \geq 2,$$ as desired. In the second case, our minimality assumption again implies $L^{(\ell-1)} \leq C-1$ since $d_z^{(\ell)} = d_z^{(\ell-1)} \geq C+1.$ This forces $d_z^{(\ell-1)} - L^{(\ell-1)} \geq 2.$ In either case, we get \[d_{\A \setminus \{z,b\}}^{(\ell-1)} = L^{(\ell-1)} - d_z^{(\ell-1)} - d_b^{(\ell-1)} \leq -C.\qedhere\]
    \end{proof}

\begin{cor}\label{cor:break}
    Assume the hypothesis of Lemma~\ref{lem:break} for an Arnoux--Rauzy word $\mathbf{x}$. Then there exists an index $\ell,$ a pair of factors $u,v$ of $\mathbf{x}^{(\ell)},$ and a pair of letters $z,b \in \A,$ such that both \ref{cond:maple} and (\ref{eq:break}) are satisfied.
\end{cor}
\begin{proof}
    By Lemma~\ref{lem:break}, there exist factors $u,v$ of $\mathbf{x}$ and an index $i_1$ such that $u^{(i_1)}, v^{(i_1)}$ satisfy (\ref{eq:break}). If $u^{(i_1)}, v^{(i_1)}$ satisfy \ref{cond:maple}, then we are done. If not, then choose an index $i_2 > i_1$ and a pair of factors $u',v'$ of $u^{(i_2)}, v^{(i_2)}$, respectively, that witness the failure of \ref{cond:maple}. Apply Lemma~\ref{lem:break} to this pair, obtaining yet another pair of factors. If this pair fails to satisfy \ref{cond:maple}, then we iterate the above process. Since failure of \ref{cond:maple} implies that the first factor in the pair has length at least $C +1 \geq 2,$ Lemma~\ref{lem:delta} implies that the first factor in the pair witnessing the failure of \ref{cond:maple} has strictly shorter length than the first factor of the preceding pair. Therefore, this process terminates after finitely many steps. The last pair must satisfy \ref{cond:maple}.\qedhere

    %Extended version: Lemma~\ref{lem:break} implies the existence of a pair of factors at a later stage which satisfy (\ref{eq:break}). If this pair of factors satisfies \ref{cond:maple}, then we win. If not, then choose a pair of factors at a later stage witnessing the failure of \ref{cond:maple}. Apply Lemma~\ref{lem:break} to this pair, to obtain yet another pair of factors at a later stage which are candidates for the success of \ref{cond:maple}. We then iterate this process. To see that this process terminates, note that the witness $u'$ to the failure of \ref{cond:maple} is a factor of $u^{(\ell)}$ for some $\ell > 0$. This, combined with the fact that $|u'| \geq |u'|_z \geq C + 1 \geq 2$, allows us to use Lemma \ref{lem:delta} to conclude $|u'| \leq |u^{(\ell)}| < |u|$ by Lemma \ref{lem:delta}. That is, each time \ref{cond:maple} fails, the first factor in our new pair is strictly shorter, meaning this process ends after finitely many steps. The last pair must satisfy \ref{cond:maple}.
\end{proof}

We now want to expand our analysis beyond a pair of letters to an arbitrarily large subset of $\A$. Given a pair of factors such as those produced by Corollary~\ref{cor:break}, as well as the assumption that $\mathbf{x}$ has bounded weak partial quotients, we introduce a lemma that tells us we have something of a conservation of energy principle regarding letter discrepancies. That is, when the discrepancy of an individual letter goes down, it must be balanced by relatively large discrepancies in other letters. 

\begin{lem}[Control Lemma]\label{lem:big} Fix an integer $C>2$ and an Arnoux--Rauzy word $\mathbf{x}$ generated by a sequence $(a_n)_{n\geq0}$ whose weak partial quotients are bounded by $h.$ Let $u,v$ be factors of $\mathbf{x}$ satisfying \ref{cond:maple}, and suppose there exist $z,b \in \A$ satisfying $(\ref{eq:break})$. Then, for every index $i,$ there exists a set $\B_i \subset \A$ such that 

\begin{equation}\label{eq:rho_your_imbalance}
\rho^{(i)}(\B_i) = - d^{(i)}_{\A \setminus \B_i} \geq (|\B_i|-1)C - R_{|\B_i|};
\end{equation} \begin{equation}\label{eq:above_T}d_a^{(i)} \geq C + R_N - R_{N+1} + 4 ~\text{for all }a \in \B_i,\end{equation} where $(R_N)_{N \geq 2}$ is the sequence defined in Theorem ~\ref{thm:lb} by $$R_2:= 0, \quad R_{N+1}:= (h+3)R_N + 2h + 13.$$
\end{lem}

Lemma~\ref{lem:big} is technical and the main innovation of this paper. We postpone its proof to the following section. In its stead, we now present the much shorter and more palatable proof of our main result.

%Preliminary readers found the proof of this lemma to be so comprehensible, straightforward, and generally pleasing that they recommended it be deferred to the next section (presumably so that later readers may have that much more time to delight in its comprehensibility etc.)

%Which we anticipate will be read with rapt attention and great enthusiasm.

\begin{proof}[Proof of Theorem~\ref{thm:lb}]
     Suppose for the sake of contradiction that the statement does not hold. Then there must exist factors of $\mathbf{x}$ satisfying the hypothesis of Lemma~\ref{lem:break}. By Corollary~\ref{cor:break}, we obtain factors $u,v$ of some Arnoux--Rauzy word satisfying the hypothesis of Lemma~\ref{lem:big}. Lemma~\ref{lem:big} then implies that, at every index $i,$ there exists a set $\B_i \subset \A$ such that $$\rho^{(i)}(\B_i) = - d^{(i)}_{\A \setminus \B_i} \geq (|\B_i|-1)C_{d,h} - R_{|\B_i|} = (N-1)\left\lceil\frac{R_d +2}{d-1}\right\rceil -R_N,$$ where $N = |\B_i|.$ It now suffices to show that
    \begin{equation}\label{eq:constant}
    (N-1)\left\lceil\frac{R_d +2}{d-1}\right\rceil -R_N> 1
    \end{equation} for all $2 \leq N \leq d,$ as this inequality would then imply that, at every index $i,$ we have $\rho^{(i)}(\B_i) > 1.$ But by Lemma~\ref{lem:delta}, we eventually have $\text{max}\{|u^{(i)}|, |v^{(i)}|\} \leq 1,$ which implies $\rho^{(i)}(\B_i) \leq 1$. This is our desired contradiction.
    
    To see why (\ref{eq:constant}) holds, we first note that, for $d = N,$ we have \[(N-1)\left\lceil\frac{R_N +2}{N-1}\right\rceil -R_N \geq R_N + 2 - R_N = 2.\] For $d \ne N \geq 2,$ we see that $$(N-1)\frac{R_{N+1}}{N} - R_N = \frac{((N-1)(h+3) - N)R_N +
    (N-1)(2h+13)}{N} > 1,$$  which in turn implies (\ref{eq:constant}).
\end{proof}

In order to leverage Theorem~\ref{thm:lb} to prove Theorem~\ref{thm:unif_fb}, we need the following result derived from \cite{espinoza2026factorbalancednesslinearrecurrencefactor}.

\begin{prop}[cf.\cite{espinoza2026factorbalancednesslinearrecurrencefactor}]\label{prop:implied}
Let $\mathbf{x}$ be an Arnoux--Rauzy word. Then the following two implications hold.
    \begin{enumerate}
        \item Suppose that $\mathbf{x}$ is uniformly factor-balanced. Then it has bounded weak partial quotients.
        \item Suppose that $\mathbf{x}$ has bounded weak partial quotients and that there exists some constant $C$ such that $\mathbf{x}^{(n)}$ is $C$-letter-balanced for each $n \geq 0$. Then $\mathbf{x}$ is uniformly factor-balanced.
    \end{enumerate}
\end{prop}

\begin{rmk}
Although not stated in this form, the proof of this result is contained in the proof of Theorem $1.4$ in \cite{espinoza2026factorbalancednesslinearrecurrencefactor}. The authors explicitly state in \cite[p.~27]{espinoza2026factorbalancednesslinearrecurrencefactor} that the first implication holds for all alphabets. Moreover, they prove that bounded weak partial quotients imply uniform factor-balancedness in the ternary case by proving the second implication in Proposition~\ref{prop:implied}, and then using Theorem $8$ from \cite{BCS} (which gives the necessary letter-balancedness condition over a ternary alphabet) to fill in the gap. In particular, it suffices to replace their invocation of Theorem $8$ from \cite{BCS} with Theorem~\ref{thm:lb} to obtain the more general statement.
\end{rmk}

\begin{proof}[Proof of Theorem~\ref{thm:unif_fb}]
    By Proposition~\ref{prop:implied}, it suffices to show that there exists a fixed constant $C$ such that $\mathbf{x}^{(n)}$ is $C$-letter-balanced for each $n \geq 0.$ This follows from the observation that $\mathbf{x}^{(n)}$ is itself an Arnoux--Rauzy word over the same
    alphabet whose generating sequence $(a_m)_{m \ge n}$ has weak partial
    quotients bounded by the same $h$ for all $n$. Therefore, Theorem \ref{thm:lb} provides us with the same letter-balance constant $C_{d,h}$ for all $\mathbf{x}^{(n)}.$ 
\end{proof}
\section{Proof of the Control Lemma}\label{sec:big}

\subsection{Outline}\label{sec:outline_big}

In this section we prove Lemma~\ref{lem:big} (the control lemma). Our proof proceeds by keeping track of the letters whose discrepancies are close to $C.$ We do so by inductively constructing a sequence of nested sets $\B_i \subset \A$ satisfying the claim of Lemma~\ref{lem:big} (which we refer to as ``big" sets). We occasionally suppress the subscript and simply write $\mathcal{B}$, particularly when our big set remains unchanged over several indices. In constructing these sets, it is helpful to define the following three thresholds, where $N$ refers to the size of $\B$: \begin{align*}
    U_N &:= C-(2R_N + 9); \\
    T_N &:= U_N - h(R_N +2) =C - (2R_N + h(R_N + 2) +9) = C - R_{N+1} + R_N +4; \\
    S_N &:= T_N - 4 = C - (2R_N + h(R_N + 2) +13) = C - R_{N+1} + R_N.
\end{align*} Note that $T_N$ is exactly the lower bound appearing in (\ref{eq:above_T}).

We briefly explain how these thresholds are derived, and their function in the following proof. The last threshold $S_N$ tells us when to enlarge $\B.$ We start with $\B_0 = \{z,b\}.$ At index $i$, suppose $|\B_{i-1}| = N.$ If $a_i \notin \B_{i-1}$ and $d_{a_i}^{(i)} \geq S_N,$ then we declare that $a_i$ has a sufficiently big discrepancy to grow our set--that is, we enlarge to $\B_{i} = \B_{i-1} \cup \{a_i\}.$ If the discrepancy of $a_i$ falls below this threshold, we keep $\B_i = \B_{i-1}.$

The main difficulty in the proof is in showing that the discrepancies of letters in $\B$ do not stray too far from $C$ between enlargements. The threshold $U_N$ is our initial upper threshold; we show that all elements of $\B$ have discrepancies above this threshold immediately succeeding an enlargement. We then want to understand what happens when the discrepancy of an element of $\B$ falls below $U_N.$ We refer to such a situation as a \textit{chaotic episode}, and the offending letter as the \textit{chaotic letter}. More precisely, a chaotic episode is an interval of indices $k_0, \dots, k_n$ such that for each $0 \leq i \leq n$ there exists a letter $c \in \B_{k_i}$ such that $d_c^{(k_i)} < U_N$ and $\B_{k_0} = \cdots =\B_{k_n}.$ In Subsection~\ref{sec:big_disc} we prove that the chaotic letter $c$ is the only element of $\B$ that can occur while this chaotic episode persists (i.e., that $a_{k_i} \in (\mathcal{A} \setminus \B_{k_i}) \cup \{c\}$ for all chaotic indices $i$). This allows us to exclusively analyze the behaviour of that single letter. 

It is this analysis that gives rise to the threshold $T_N.$ We consider the initial run of the chaotic letter at the beginning of the chaotic episode. We show that, at each step within that run, its discrepancy can decrease by at most $R_N + 2.$ By our assumption that the weak partial quotients are bounded by $h,$ this initial run has length at most $h.$ This implies that the discrepancy of the chaotic letter cannot go below $U_N - h(R_N +2) = T_N$ during the initial run. This first fall actually provides a lower bound for the chaotic letter's discrepancy throughout the chaotic episode; in particular, we show that its discrepancy can only fall below $T_N$ if the discrepancy of a letter outside of $\B$ has risen above $S_N.$

Finally, note that our recursive definition of the sequence $R_N$ can be rewritten $R_{N+1} = R_N - S_N + C.$ This is how the sequence $R_N$ is derived. As shown in subsection \ref{sec:big_bigger}, this follows naturally from our admission criterion.

The argument is divided into sub-sections as follows. Subsection \ref{sec:big_bigger} handles the base case, as well as the indices at which we grow our big set--in particular it shows that all elements of the big set $\B$ have discrepancies above $U_N$ and that (\ref{eq:rho_your_imbalance}) still holds after enlargement. Subsection \ref{sec:big_disc} establishes that, for fixed $N,$ all elements of $\B$ have discrepancies above $T_N,$ assuming (\ref{eq:rho_your_imbalance}) at smaller indices. Subsection \ref{sec:small_smaller} establishes (\ref{eq:rho_your_imbalance}) at index $n+1$, assuming that all elements of $\B$ have discrepancies above $T_N$ at smaller indices.

\subsection{When our big set gets bigger}\label{sec:big_bigger}

We start by constructing our first big set and verifying that it has the big properties we want. If either $a_0 \in \{z,b\}$ or $d_{a_0}^{(0)} < S_2,$ then define $\B_0 := \{z,b\}.$ By assumption, we have $$\rho_0(\{z,b\}) \geq C = (|\{z,b\}|-1)C - R_2 \quad \text{and}\quad d_z^{(0)}, d_b^{(0)} \geq C -2 \geq U_2.$$

If, on the other hand, $a_0 \notin \{z,b\}$ and $d^{(0)}_{a_0} \geq S_2,$ then we define $\B_0 := \{z,b,a_0\}.$ We similarly verify that 
$$\rho_0(\{z,b,a_0\}) \geq C + S_2 = 2C - R_3$$ and that the discrepancies of $z,b,$ and $a_0$ fall above $U_3.$ This concludes the base case.

Now suppose we have a big set $\B_{n-1}$ that satisfies (\ref{eq:rho_your_imbalance}) and (\ref{eq:above_T}), and that we admit $a_n$ to this set (meaning that $d_{a_n}^{(n)} \geq S_N$). Our recursive definition of $R_{N+1}$ immediately tells us that

\begin{align*}
    \rho_n(\B_{n-1} \cup \{a_n\})=\rho_n(\B_{n-1}) + d^{(n)}(a_n) \geq (N-1)C - R_N + S_N = NC - R_{N+1}.
\end{align*} Additionally, we readily see that since $U_{N+1} < \text{min}\{T_N, S_N\},$ every letter in $\B_n$ satisfies our upper threshold. Therefore, to prove Lemma~\ref{lem:big} we now just need to show that (\ref{eq:above_T}) (the lower threshold for discrepancies of letters in our big set) and (\ref{eq:rho_your_imbalance}) (the lower bound on collective discrepancies of letters outside our big set) are satisfied at indices where we do not grow our big set.

\subsection{Keeping big discrepancies big}\label{sec:big_disc}

In this subsection, we prove that (\ref{eq:above_T}) holds for all letters in our big set. We do so by analyzing the behaviour of the discrepancies of letters in our big set when one of the letters' discrepancies falls below our upper threshold. As noted above, we say that an index $i$ is \textit{chaotic} if there exists a letter $c \in \B_i$ such that $d^{(i)}_c < U_N$ and refer to an interval of chaotic indices as a \textit{chaotic episode}. Note that at every non-chaotic index, (\ref{eq:above_T}) holds by definition. It therefore suffices to look at what happens at chaotic indices. 

First, we introduce an identity that describes how successive occurrences of a single letter affect that letter's discrepancy. Suppose we have a run $c = a_\ell = \dots = a_{\ell + m - 1}$ starting at some index $\ell.$ Recursive application of (\ref{eq:d}) then gives us

\begin{equation}\label{eq:rec_d}
    d_c^{(\ell+m)} = (m+1)d_c^{(\ell)} - (m)L^{(\ell)} +  \sum_{j=\ell}^{\ell+m-1}(\delta_{u}^{(j)}-\delta_{v}^{(j)}). 
\end{equation} We highlight this equation because its application will be the only point in the proof where we leverage the assumption that the weak partial quotients of $\mathbf{x}$ are bounded by $h$ (meaning, in the context of (\ref{eq:rec_d}), that $m \leq h$).  

Suppose now that $t+1$ is the smallest index of a given chaotic episode, and write $a_t = c.$ The rest of the argument is divided into three central claims. The first claim is a tool to prove the second two, which then establish (\ref{eq:above_T}).

\begin{claim}\label{claim:below_C} 
The length discrepancy $L^{(t+1)}$ is strictly below $C.$ By \ref{cond:maple}, this implies that $d_e^{(t+1)} \leq C$ for all letters $e \in \B \setminus \{c\}$.
\end{claim}

\begin{proof}

Suppose the claim does not hold, i.e., that $L^{(t+1)} \geq C$. Set $$\alpha = L^{(t+1)} - C \geq 0.$$ We then rewrite (\ref{eq:sum}) as

$$C+ \alpha = L^{(t+1)} = d_{\B \setminus \{c\}}^{(t+1)} + d_c^{(t+1)} - \rho_{\B}^{(t+1)}.$$

This, alongside  our inductive assumption that $\rho_\B^{(t+1)} \geq (N-1)C - R_N$ and our assumption that $c$ is chaotic at $t+1$, gives us the lower bound

\begin{align*}
    (N-1)\cdot C -R_N + C + \alpha - U_N + 1 = (N-1)\cdot C + \alpha + R_N + 10 \leq d_{\B \setminus \{c\}}^{(t+1)}.
\end{align*} Note that if we truncate $u^{(t+1)}$ by removing the first $\alpha + 1$ letters, we obtain a factor $u'$ such that $|u'| - |v^{(t+1)}| = C-1$ and $$|u'|_{\B \setminus \{c\}} - |v^{(t+1)}|_{\B \setminus \{c\}} \geq (N-1)\cdot C + 9 + R_N > (N-1)C + 1.$$ By the pigeonhole principle there must exist some $e \in \B \setminus \{c\}$ such that $|u'|_e - |v^{(t+1)}|_e \geq C+1,$ which contradicts \ref{cond:maple}.
\end{proof}

\begin{claim}\label{claim:only_c} The chaotic letter $c \in \B$ is the only letter in $\B$ that can occur during the chaotic episode.
\end{claim}

\begin{proof}
We first construct a lower bound on the discrepancies of letters $e \in \B \setminus \{c\}$. Set $\beta = d_c^{(t)} - d_c^{(t+1)} \geq 1.$ By (\ref{eq:sum}) and~(\ref{eq:L}), we have \[
    d_c^{(t)} - 2 \leq L^{(t+1)} \leq d_c^{(t+1)} + d_{\B \setminus \{c\}}^{(t+1)}  - (N-1)C + R_N,\] which implies \[d_{\B \setminus \{c\}}^{(t+1)} \geq (N-1)C -R_N -2 + \beta.\]We then leverage Claim~\ref{claim:below_C} to note that, for any $e \in \B \setminus \{c\},$ \[d_e^{(t+1)} + (N-2)C \geq d_{\B \setminus \{c\}}^{(t+1)} \geq (N-1)C -R_N -2 + \beta,\] and so \[d_e^{(t+1)} \geq C - R_N - 2 +\beta.\]

We now use this lower bound to prove the claim inductively. We start with our base case, showing that if the chaotic episode persists after our first run of $c$'s, then the next letter must be in $\A \setminus \B.$ Suppose the run has length $m,$ and let $k$ be the letter at index $t+m$. Suppose for contradiction that $k \in \B.$ Then $d_c^{(t+m)} < U_N  = C - 2R_N - 9$ implies that

\begin{align*}
    d_k^{(t+m+1)} &\geq 2d_k^{(t+1)} - L^{(t+m)} + \delta_u - \delta_v  \\
    &\geq 2d_k^{(t+1)} - (d_c^{(t+m)} - d_c^{(t)} + L^{(t)}) -2 \\
    &\geq 2C - 2R_N + 2\beta - 4 - (C - 2R_N  -10 + \beta + 2) -2   \\
    &= C +\beta + 2 \\
    &\geq C + 1,
\end{align*} where the first, second, and third inequalities follow from (\ref{eq:d}), (\ref{eq:stable:diff}), and (\ref{eq:diff}), respectively. Moreover, we have that 

\begin{align*}
d_k^{(t+m+1)} - L^{(t+m+ 1)} &= d_k^{(t+m)} -L^{(t+m)}\\ &\geq d_k^{(t+1)} - (d_c^{(t+m)} - d_c^{(t)} + L^{(t)}) \\
&\geq C - R_N - 2 + \beta -C + 2R_N +9  + 1 -\beta - 2 \\
&\geq R_N + 6 > 2.
\end{align*} By Remark~\ref{rmk:maple}, this contradicts \ref{cond:maple}.

Suppose now that we have a chaotic index $p > t+m$ such that each intermediary index $t+m \leq i < p$ is chaotic persists and our inductive hypothesis $a_i \in (\A \setminus \B) \cup \{c\}$ holds. If $a_p \in \B,$ we want to show that $a_p = c.$ Our inductive assumption implies that $d_{a_{p-1}}^{(p-1)} < U_N.$ If $a_{p-1} = c,$ then this follows by definition of the chaotic episode. If $a_{p-1} \notin \mathcal{B},$ then $d_{a_{p-1}}^{(p-1)} \geq U_N > S_N$ would trigger an enlargement of $\mathcal{B}$ and end the chaotic episode. It follows that $L^{(p)} \leq d_{a_{p-1}}^{(p-1)} + 2 \leq U_N +1 = C -2R_N - 8.$ We then have from (\ref{eq:d}) that

\begin{equation}\label{eq:a_p}
    d_{a_p}^{(p+1)} \geq 2d_{a_p}^{(p)} - C + 2R_N +6.
\end{equation}

If $a_p \in \B \setminus\{c\},$ then our inductive assumption that it did not occur at any index $t \leq i < p$ implies that $d_{a_p}^{(p)} = d_{a_p}^{(t+1)} \geq C - R_N - 2 +\beta \geq C - R_N -1$. Plugging this bound into (\ref{eq:a_p}), we obtain

\begin{align*}
    d_{a_p}^{(p+1)} &\geq 2(C - R_N -1 ) - C + 2R_N +9 -3 \\
    &\geq C+ 4.
\end{align*}

On the other hand, our inductive assumption also states that our upper bound $d_{a_p}^{(p)} \leq C$ still holds, meaning that $L^{(p+1)} \leq C+2.$ By Remark~\ref{rmk:maple}, we conclude that $a_p \in \B \setminus\{c\}$ contradicts \ref{cond:maple}.
\end{proof}

\begin{claim}\label{claim:escape} $d_c^{(i)} \geq T_N$ for all indices $i \geq t$ during this chaotic episode.
\end{claim}
\begin{proof} As before, let $m$ be the number of $c$'s that occur after index $t.$ We first show that $d_c^{(t+\ell)} \geq T_N$ for all $1 \leq \ell \leq m$ and then use this to show inductively that the discrepancy of $c$ stays above this threshold for the duration of the chaotic episode. Set $\gamma := L^{(t)} - d_c^{(t)}$. By (\ref{eq:sum}), we have that

\begin{align*}
    \gamma = L^{(t)} - d_c^{(t)} &\leq d_{\B \setminus \{c\}}^{(t)} - (N-1)C + R_N,
\end{align*} and so $(N-1)C+ \gamma - R_N\leq d_{\B \setminus \{c\}}^{(t)}.$ By Claim~\ref{claim:below_C}, $d_{\B \setminus \{c\}}^{(t+1)} \leq (N-1)C.$ Since $d_{\B \setminus \{c\}}^{(t)} = d_{\B \setminus \{c\}}^{(t+1)},$ we can conclude from the above that $$(N-1)C+ \gamma - R_N\leq (N-1)C,$$ which implies $\gamma \leq R_N.$ By hypothesis, the weak partial quotients of $\mathbf{x}$ are bounded by $h$, and so $m \leq h$. Using (\ref{eq:rec_d}) and our bound $\gamma \leq R_N,$ we conclude

\begin{align*}
    d_c^{(t+\ell)} &= (\ell+1)d_c^{(t)} - \ell\cdot L^{(t)} + \sum_{j=t}^{t + \ell-1}(\delta_{u}^{(j)}-\delta_{v}^{(j)}) \\
    &\geq d_c^{(t)} + \ell(d_c^{(t)} - L^{(t)}) -2\ell\\
    &\geq U_N - \ell\cdot \gamma - 2\ell\\
    &\geq C - 2R_N -h (R_N + 2) - 9 \\
    &= T_N,
\end{align*}for all $\ell \leq m$, as desired.

Consider now an index $p > t$ in our chaotic episode starting a run of $c$'s $a_{p-1} \ne a_p = \dots = a_{p+r-1} = c$ of length $r.$ Having dealt with our base case, we now suppose inductively that $d_c^{(i)} \geq T_N$ for all $t \leq i \leq p$. By Claim~\ref{claim:only_c}, $a_{p-1} \notin \B,$ which implies that $d_{a_{p-1}}^{(p-1)} < S_N.$ Thus $L^{(p)} \leq S_N +1 = T_N - 3,$ and (\ref{eq:diff}) implies that, at each index $i$ with $p \leq i <p+r,$ we have 

\begin{align*}
    d_c^{(i+1)} - d_c^{(i)} &\geq d_c^{(p)} - L^{(p)} -2 \\
    &\geq T_N - T_N +3 -2  = 1,
\end{align*} where the second inequality follows from the inductive hypothesis that $d_c^{(p)} \geq T_N.$

This means that either the chaotic episode ends during this run of $c$'s, or the run is followed by a letter in $\A \setminus \B.$ In the latter case, the lower bound of $T_N$ on the discrepancy of $c$ still holds.
\end{proof}

\subsection{Keeping small discrepancies small}\label{sec:small_smaller}
Our last step is to show that (\ref{eq:rho_your_imbalance}) holds for the indices $i$ such that $|\B_i| = |\B_{i+1}|.$ If $a_n \in \B_n$, then $\rho_{n+1}(\B_n) = \rho_n(\B_n)$ by definition of $\rho$. 

Suppose now that $a_n \notin \B_n.$ We look to the greatest index $\ell < n$ such that $a_\ell \in \B_n$, if such an index exists (we deal with the case where it does not below). This is where we leverage (\ref{eq:above_T}), which was established in the previous section. In particular, this bound tells us that

$$L^{(\ell+1)} \geq d^{(\ell)}_{a_\ell} - 2\geq T_N - 2 = C - 2R_N - h(R_N + 2) -11,$$ where $N = |\B_n|.$ We additionally note that $a_n \notin \B_n$ implies that $d^{(n)}_{a_n} < C - S_{N},$ and so (\ref{eq:L}) tell us that

$$L^{(n+1)} \leq d_{a_n}^{(n)} +2 \leq S_N +1 = C - 2R_N - h(R_N + 2) - 12 < L^{(\ell + 1)}.$$ Since no letter in our big set occurs between indices $\ell$ and $n,$ the inequality above implies \[\rho_{n+1}(\B) = d^{(n+1)}(\B) - L^{(n+1)} 
    > d^{(\ell +1)}(\B) - L^{(\ell+1)} 
    = \rho_{\ell+1}(\B),\] and so (\ref{eq:rho_your_imbalance}) is maintained during occurrences of letters outside our big set.

Finally, we address the case where none of the $a_\ell$ from $\ell < n$ are in $\B$. This can only happen if $\B_0 = \{z,b\}$ and $a_0 \notin \B_0,$ and we are looking at an index $i$ prior to the first occurrence of either $z$ or $b.$ But in this case, we have $d_{\{z,b\}}^{(i+1)} \geq 2C -1$ by (\ref{eq:break}), while $$L^{(i+1)} \leq d_{a_i}^{(i)} + 2  \leq C-S_2 + 1 = C - 2h -12 < C-1,$$ since $a_i \notin \B.$ Therefore $$\rho_{i+1}(\B_0) = d^{(i+1)}(\B_0) - L^{(i+1)} > C,$$ which satisfies (\ref{eq:rho_your_imbalance}). This completes our proof of the lemma.

\section{Avenues for Further Exploration}
It would be interesting to give characterizations of letter-balancedness and factor-balancedness analogous to the characterization of uniform factor-balancedness given in Theorem~\ref{thm:unif_fb}. As shown in \cite{BCS}, there exist letter-balanced Arnoux--Rauzy words with unbounded weak partial quotients. Intuitively, any condition that applies uniformly to the entire generating sequence (like bounding its weak partial quotients) is too strong, as it would correspondingly give the same balance constant for each $\mathbf{x}^{(m)}.$ By Proposition~\ref{prop:implied}, this would imply uniform factor-balancedness. One could instead consider under what circumstances there exists an unbounded sequence of letter-balance constants $C_m$ for the $\mathbf{x}^{(m)}$. In particular, if $(k_N)_{N \geq 0}$ is the sequence of weak partial quotients associated to $(a_n)_{n \geq 0},$ one could require that $k_N \leq f(N)$ for a non-constant function  $f: \N \to \N$. Choosing $f$ linear or polynomial could loosen the restriction enough to obtain a more general factor-balancedness result. 

%Alternatively, one could look into different feature of the generating sequence altogether (for example ...

We also note that the constants $C_{d,h}$ in Theorem~\ref{thm:lb} are far from optimal. For example, we consider the ternary constant $C_{3,h} = h+8$. Asymptotically, this is an improvement on the letter-balance constant $2h +1$ from \cite{BCS}, but separate work suggests that $h+3$ is the best one can do. In particular, given the rate at which the constants $C_{d,h}$ grow in $h,$ we suspect that it is possible to vastly improve the letter-balance constants.

\section{Acknowledgements}

This research was conducted at the University of Minnesota Duluth REU, which was funded by Jane Street Capital, Axiom Math AI, and a donation from Larry Penn.
		
I am very grateful to Joe Gallian and Colin Defant for providing this incredible opportunity, as well as to advisors Eliot Hodges, Claire Kaneshiro, Noah Kravitz, and Carl Schildkraut. Special thanks as well to Glenn Bruda for his thoughtful comments throughout.

\section{AI Usage Statement}
The writing and ideas contained in this paper are the author's original work. Generative AI (specifically, ChatGPT Plus) played the role of a kind of mathematically sophisticated ``rubber duck" during the research process. Early formulations of Lemma~\ref{lem:big} (the control lemma) were suggested to ChatGPT, accompanied by general prompts asking what the AI model thought about the statements' mathematical viability. The formalizations that ChatGPT proposed in response were not themselves usable -- they contained many false assertions and over-optimistic conjectures. It was, rather, the process of explaining these mistakes and their potential resolutions to ChatGPT that proved helpful in coming up with the construction of the big sets $\B$ as well as the relevant thresholds. ChatGPT was also used to find errors (both typographical and mathematical) in preliminary drafts.
\bibliography{ref}

@article{Cassaigne2000,
author = {Julien Cassaigne and Sébastien Ferenczi and Luca Q. Zamboni},
journal = {Annales de l'institut Fourier},
language = {eng},
number = {4},
pages = {1265-1276},
publisher = {Association des Annales de l'Institut Fourier},
title = {Imbalances in {A}rnoux--{R}auzy sequences},
url = {http://eudml.org/doc/75456},
volume = {50},
year = {2000},
}

@article{Arnoux1991,
author = {Pierre Arnoux and Gérard Rauzy},
journal = {Bulletin de la Société Mathématique de France},
language = {fre},
number = {2},
pages = {199-215},
publisher = {Société mathématique de France},
title = {Représentation géométrique de suites de complexité $2n+1$},
url = {http://eudml.org/doc/87622},
volume = {119},
year = {1991},
}

@misc{espinoza2026factorbalancednesslinearrecurrencefactor,
      title={Factor-balancedness, linear recurrence, and factor complexity}, 
      author={Bastiàn Espinoza and Pierre Popoli and Manon Stipulanti},
      year={2026},
      eprint={2602.03746},
      archivePrefix={arXiv},
      primaryClass={math.CO},
      url={https://arxiv.org/abs/2602.03746}, 
}

@inbook{Delecroix_2013,
   title={Balancedness of Arnoux--Rauzy and Brun Words},
   ISBN={9783642405792},
   ISSN={1611-3349},
   url={http://dx.doi.org/10.1007/978-3-642-40579-2_14},
   DOI={10.1007/978-3-642-40579-2_14},
   booktitle={Combinatorics on Words},
   publisher={Springer Berlin Heidelberg},
   author={Delecroix, Vincent and Hejda, Tomáš and Steiner, Wolfgang},
   year={2013},
   pages={119–131} }

@article{Morse1940SymbolicDI,
  title={Symbolic {D}ynamics {II.} {S}turmian {T}rajectories},
  author={Marston Morse and Gustav Arnold Hedlund},
  journal={American Journal of Mathematics},
  year={1940},
  volume={62},
  pages={1},
  url={https://api.semanticscholar.org/CorpusID:124432686}
}

@incollection{thuswaldner2020sadics,
  author    = {J{\"o}rg M. Thuswaldner},
  title     = {{$S$}-adic sequences: A bridge between dynamics, arithmetic, and geometry},
  booktitle = {Substitution and Tiling Dynamics: Introduction to Self-inducing Structures},
  editor    = {Shigeki Akiyama and Pierre Arnoux},
  series    = {Lecture Notes in Mathematics},
  volume    = {2273},
  pages     = {97--191},
  publisher = {Springer},
  address   = {Cham},
  year      = {2020},
  doi       = {10.1007/978-3-030-57666-0_3},
}

@article{berthe2021multidimensionalcontinuedfractionssymbolic,
title = {Multidimensional continued fractions and symbolic codings of toral translations},
journal = {Journal of the European Mathematical Society},
volume = {25},
number = {12},
pages = {4997–5057},
year = {2023},
doi = {10.4171/JEMS/1300},
url = {https://ems.press/journals/jems/articles/8696283},
author = {Valérie Berthé and Wolfgang Steiner and Jörg Thuswaldner}
}

@article{BERTHE201993,
title = {Balancedness and coboundaries in symbolic systems},
journal = {Theoretical Computer Science},
volume = {777},
pages = {93-110},
year = {2019},
note = {In memory of Maurice Nivat, a founding father of Theoretical Computer Science - Part I},
issn = {0304-3975},
doi = {https://doi.org/10.1016/j.tcs.2018.09.012},
url = {https://www.sciencedirect.com/science/article/pii/S0304397518305772},
author = {Valérie Berthé and Paulina {Cecchi Bernales}}
}

@article{DBLP:journals/mst/CovenH73,
  author       = {Ethan M. Coven and
                  Gustav Arnold Hedlund},
  title        = {Sequences with Minimal Block Growth},
  journal      = {Math. Syst. Theory},
  volume       = {7},
  number       = {2},
  pages        = {138--153},
  year         = {1973},
  url          = {https://doi.org/10.1007/BF01762232},
  doi          = {10.1007/BF01762232},
  bibsource    = {dblp computer science bibliography, https://dblp.org}
}

@article{BCS,
author = {Val{\'{e}}rie Berth{\'{e}} and
                  Julien Cassaigne and
                  Wolfgang Steiner},
title = {Balance properties of {A}rnoux--{R}auzy words},
journal = {International Journal of Algebra and Computation},
volume = {23},
number = {04},
pages = {689-703},
year = {2013},
doi = {10.1142/S0218196713400043},

URL = { 
    
        https://doi.org/10.1142/S0218196713400043
    
    

},
eprint = { 
    
        https://doi.org/10.1142/S0218196713400043
    
    

}
}
\bibliographystyle{amsplain}
\end{document}